\documentclass[11pt]{article}
\usepackage{amsmath,amsthm,amsfonts,amssymb,amscd, amsxtra,color}
\usepackage{graphicx,multirow}
\usepackage{float}
\usepackage{xcolor}
\usepackage[active]{srcltx}
\usepackage{url}
\usepackage{cite}
\usepackage[margin=1.0in]{geometry}
\newtheorem{theorem}{Theorem}[section]

\newtheorem{lemma}{Lemma}[section]
\newtheorem{proposition}{Proposition}[section]
\newtheorem{remark}{Remark}[section]

\newtheorem{definition}{Definition}[section]

\def\dom{\operatorname{dom}}

\def\argmin{\operatorname{argmin}}

\def\min{\operatorname{min}}

\def\min{\operatorname{Minimize}}

\title{An extragradient-type algorithm for variational inequality\\  on Hadamard manifolds}
\author{Batista,  E. E. A.\thanks{CCET, Universidade Federal do Oeste da Bahia,
Barreiras, BA 47808-021, BR ({\tt edvaldo.batista@ufob.edu.br}).}
\and Bento,  G. C.\thanks{IME, Universidade Federal de Goi\'as,
Goi\^ania, GO 74001-970, BR ({\tt glaydston@ufg.br}).}
\and
Ferreira, O. P.
\thanks{IME, Universidade Federal de Goi\'as,
Goi\^ania, GO 74001-970, BR({\tt orizon@ufg.br}).}
}
\begin{document}
\maketitle
\begin{abstract}
The aim of this paper is to  present an extragradient method  for variational inequality  associated to a point-to-set vector field  in Hadamard manifolds and to study its convergence properties.  In order to present our method  the concept of $\epsilon$-enlargement  of maximal monotone vector fields  is used   and its  lower-semicontinuity  is stablished in order to obtain   the convergence of the method in this new context.\\

\noindent
{\bf Keywords:} Extragradient algorithm; Hadamard manifolds; $\epsilon$-enlargement; lower-semicontinuity.

\end{abstract}
\section{Introduction}

In its  classical formulation,  the variational inequality problem is an inequality involving a operator, which has to be solved for all possible values of a given variable belonging  to a convex set in a  liner space.   As well known  the variational inequality problem is an abstract model for various problems in classical analysis and its applications, where the  convexity and linear structure plays an important role.  However, in many practical applications  the natural structure of the data are modeled as constrained optimization problems, where the constraints  are  non-linear and non-convex, more specially, the constraints are  Riemannian manifolds; see \cite{AdlerDedieuShub2002, BacakBergmannSteidl2016, DasChakrabortiChaudhuri2001, DreznerWesolowsky1983, BergmannPerschSteidl2016, BergmannWeinmann2016, BhattacharyaPatrangenaru2003, BhattacharyaPatrangenaru2005, BhattacharyaBhattacharya2008, Fletcher2013, Pennec2006, Freifeld2012, BergmannWeinmann2016_2, kleinsteuber2012, HaweKleinsteuberDiepold2013, Smith1994}. Due to  such applications, interest in the development of optimization tools as well as mathematical programming methods to Riemannian settings has increased significantly;  papers published on  this topic  include, but are not limited to,   \cite{WangLiWangYao2015, SouzaOliveira2015, BentoCruzNetoOliveira2016, ChenHuang2016, SuparatulatornCholamjiak2017, WangLiLopez2016, RafaNicolae2016,  AhmadiKhatibzadeh2014, Bacak2013, BentoFerreira2015,  TangHuang2013, WangLiLopezYao2015, CruzNetoSantosSoares2016, GrohsHosseini2016-2}.

In this paper, we consider the problem of  finding a  solution of a variational inequality problem on Riemannian context, this problem   were first introduced and studied  by  N\'emeth in \cite{Nemeth2003},  for univalued vector fields on Hadamard manifolds,   and  for multivalued  vector fields on general Riemannian manifolds  by Li and Yao in   \cite{LiYao2012};  for recent works addressing this subject  see \cite{FangChen2015, LiChongLiouYao2009, TangWangLiu2015, TangZhouHuang2013}. It is worth noting that  constrained optimization problems  and the problem of finding the zero of a  multivalued vector field,  which were  studied  in \cite{AhmadiKhatibzadeh2014, BentoFerreira2015, daCruzFerreiraPerez2006,  FerreiraOliveira2002, LiLopesMartin-Marquez2009, WangLiLopezYao2015},  are  particular instances of the variational inequality.

The aim of this paper is to  present an extragradient-type algorithm for variational inequality  associated to a point-to-set vector filed  in Hadamard manifolds and to study its convergence properties. In order to present our method we utilize the concept of $\epsilon$-enlargement   introduced  by \cite{BurachikIusemSvaiter1997} in Euclidean spaces and generalized by \cite{BatistaBentoFerreira2015-2} to the Riemannian context.  It is worth mentioning that the concept of $\epsilon$-enlargement   in linear spaces has been successfully employed for a wide range of purposes; see \cite{BurachikIusem2008} and its reference therein. In particular,   the $\epsilon$-enlargement was used  to establish the iteration complexity of the  hybrid proximal extragradient method, see \cite{MonteiroSvaiter2010}.  The convergence analysis of the  extragradient algorithm  in linear setting associated to a point-to-point operator  cannot be automatically extended to point-to-set operator. As was remarked in \cite{IusemPerez2000}, the reason for this failure is the lack of lower-semicontinuity of arbitrary maximal monotone operator.  In \cite{IusemPerez2000} was showed that  the $\epsilon$-enlargement of  arbitrary   maximal monotone operators are lower-semicontinuity and then  extended the extragradient algorithm to point-to-set operator.   In this paper we show that lower-semicontinuity remains valid for the   $\epsilon$-enlargement of arbitrary maximal monotone vector fields  and this allowed us to prove the convergence of the extragradient algorithm for variational inequalities associated to a point-to-set vector filed  in Hadamard manifolds.  Finally, we  state that  the proposed method   has two  important particular instances, namely, the methods (3.1) of \cite{IusemPerez2000} and (4.1) of \cite{TangHuang2012}.

The remainder of this paper is organized as follows. In Section~\ref{sec2},  some notations and  basic results used in the paper are presented. In Section~\ref{sec3}, the  concept of enlargement of monotone vector fields is introduced and  some properties are obtained. In Section \ref{sec4}, an extragradient algorithm for variational inequalities is presented  and its convergence properties are studied. Finally, the conclusions are presented in Section~\ref{sec5}.
\subsection{Notation and terminology}
\label{sec2}

In this section, we introduce some fundamental  properties and notations about Riemannian geometry. These basics facts can be found in any introductory book on Riemannian geometry, such as in \cite{doCarmo1992} and \cite{Sakai1996}.

Let $M$ be a $n$-dimentional Hadamard  manifold. {\it In this paper, all manifolds $M$ are assumed to be Hadamard finite dimensional}. We denote by $T_pM$ the $n$-dimentional {\it tangent space} of $M$ at $p$, by $TM=\cup_{p\in M}T_pM$ {\itshape{tangent bundle}} of $M$ and by ${\cal X}(M)$ the space of smooth vector fields on $M$. The Riemannian metric is denoted by  $\langle \,,\, \rangle$ and  the corresponding norm  by $\| \; \|$. Denote the lenght of piecewise smooth curves $\gamma:[a,b]\rightarrow M$ joining $p$ to $q$, i.e., such that $\gamma(a)=p$ and $\gamma(b)=q$, by
\[
\ell(\gamma)=\int_a^b\|\gamma^{\prime}(t)\|dt,
\]
and the Riemannian  distance by $d(p,q)$,  which induces the original topology on $M$, namely,  $( M, d)$ is a complete metric space and bounded and closed subsets are compact.  For $A\subset M$,  the notation  $ \mbox{int} (A)$ means the interior of the set $A$, and if $A$ is a nonempty set,  the distance  from $p\in M$ to $A$ is given by $d(p,A):= \inf \{d(p,q)~:~ q\in A\}.$   Let $\nabla$ be the Levi-Civita connection associated to $(M,{\langle} \,,\, {\rangle})$. A vector field $V$ along $\gamma$ is said to be {\it parallel} if $\nabla_{\gamma^{\prime}} V=0$. If $\gamma^{\prime}$ itself is parallel we say that $\gamma$ is a {\it geodesic}. Given that geodesic equation $\nabla_{\ \gamma^{\prime}} \gamma^{\prime}=0$ is a second order nonlinear ordinary differential equation, then geodesic $\gamma=\gamma _{v}(.,p)$ is determined by its position $p$ and velocity $v$ at $p$. It is easy to check that $\|\gamma ^{\prime}\|$ is constant. We say that $ \gamma $ is {\it normalized} if $\| \gamma ^{\prime}\|=1$. The restriction of a geodesic to a  closed bounded interval is called a {\it geodesic segment}. Since $M$ is a Hadamard manifolds the lenght   of the  geodesic segment  $\gamma$  joining $p$ to $q$ its equals $d(p,q)$, the parallel transport along $\gamma$ from $p$ to $q$ is denoted by $P_{pq}:T_{p}M\to T_{q}M$. Moreover, {\it exponential map} $\exp_{p}:T_{p}  M \to M $ is defined by $\exp_{p}v\,=\, \gamma _{v}(1,p)$ is  a diffeomorphism and, consequently, $M$ is diffeomorphic to the Euclidean space $\mathbb{R}^n $, $ n=dim M $.  Let ${q}\in M $ and $\exp^{-1}_{q}:M\to T_{p}M$ be the inverse of the exponential map.  Set $ d_{q}(p):=d(q,p)$ and note that $d_{q}(p)=\|\exp^{-1}_{p}q\|$.    Furthermore,  we know that
\begin{equation} \label{eq:coslaw}
d^2(p_1,p_3)+d^2(p_3,p_2)-2\left\langle \exp_{p_3}^{-1}p_1,\exp_{p_3}^{-1}p_2\right\rangle\leq d^2(p_1,p_2),  \qquad p_1, p_2 , p_3 \in M.
\end{equation}
\begin{equation} \label{eq:coslaw2}
\left\langle \exp^{-1}_{p_2}p_1, \, \exp^{-1}_{p_2}p_3\right\rangle+\left\langle \exp^{-1}_{p_3}p_1, \, \exp^{-1}_{p_3}p_2\right\rangle \geq d^2(p_2,p_3),  \ \qquad p_1, p_2 , p_3 \in M.
\end{equation}
A set,  $\Omega\subseteq M$ is said to be {\it convex}  if any geodesic segment with end points in $\Omega$ is contained in
$\Omega$, that is,  if $\gamma:[ a,b ]\to M$ is a geodesic such that $x =\gamma(a)\in \Omega$ and $y =\gamma(b)\in \Omega$; then $\gamma((1-t)a + tb)\in \Omega$ for all $t\in [0,1]$.  Let $\Omega\subset M$ be a convex set and $p\in M$. Thus the  {\it projection} ${\rm P}_{\Omega}(p):=\{ \tilde{q}\in \Omega ~: ~ d(\tilde{q}, p)\leq d(q, p),~ q\in \Omega\}$ of $p$ onto $\Omega$ satisfies
\begin{equation} \label{eq:proj}
\left\langle \exp^{-1}_{{\rm P}_{\Omega}(p)}q, \, \exp^{-1}_{{\rm P}_{\Omega}(p)}p\right\rangle \leq 0, \qquad q\in \Omega,
\end{equation}
see \cite[Corollary 3.1]{FerreiraOliveira2002}. The metric projection onto a nonempty, closed and convex subset $\Omega\subset M$ is a firmly nonexpansive mapping; see \cite[Corollary 1]{LiLopesMartinWang2011}. As a consequence, the projection mapping  is nonexpansive, i.e.  there holds
\begin{equation} \label{eq:projnon-expansive}
 d({\rm P}_{\Omega}(p), {\rm P}_{\Omega}(q)) \leq d(p,q)  \qquad p,  q\in \Omega.
\end{equation}
\begin{lemma} \label{eq:ContExp}
Let $\bar{p},\bar{q}\in M$ and $\{p^k\},\{q^k\}\subset M$ be such that $\lim_{k\to +\infty} p^k=\bar{p}$ and $\lim_{k\to +\infty}q^k=\bar{q}$. Then the following assertions hold.
\begin{itemize}
\item[i)] For any $q\in M$,  $\lim_{k\to +\infty} \exp^{-1}_{p^k}q=  \exp^{-1}_{\bar{p}}q$ and $\lim_{k\to +\infty} exp^{-1}_qp^k= \exp^{-1}_q\bar{p}$;
\item[ii)] If $v^k\in T_{p^k}M$ and $\lim_{k\to +\infty}  v^k=\bar{v}$, then $\bar{v}\in T_{\bar{p}}M$;
\item[iii)] For any $u\in T_{\bar{p}}M$, the function $F:M\rightarrow TM$ defined by $F(p)=\rm P_{\bar{p}p}u$ each $p\in M$ is continuous on $M$;
\item[iv)] $\lim_{k\to +\infty}  \exp^{-1}_{p^k}q^k=\exp^{-1}_{\bar{p}}\bar{q}$.
\end{itemize}
\end{lemma}
\begin{proof}
For items (i), (ii) and (iii) see \cite[Lemma 2.4]{LiLopesMartin-Marquez2009}. For the item (iv), using triangular inequality we obtain
\begin{equation*} \parallel \exp^{-1}_{p^k}q^k-\rm P_{\bar{p}p^k}\exp^{-1}_{\bar{p}}\bar{q}\parallel\leq \parallel \exp^{-1}_{p^k}q^k-\exp^{-1}_{p^k}\bar{q}\parallel+\parallel \exp^{-1}_{p^k}\bar{q}-\rm P_{\bar{p}p^k}\exp^{-1}_{\bar{p}}\bar{q}\parallel.
\end{equation*}
Since $M$ have nonpositive curvature we have $\parallel \exp^{-1}_{p^k}q^k-\exp^{-1}_{p^k}\bar{q}\parallel\leq d(q^k,q)$. It follows then that
\begin{equation*} \parallel \exp^{-1}_{p^k}q^k-\rm P_{\bar{p}p^k}\exp^{-1}_{\bar{p}}\bar{q}\parallel\leq d(q^k,\bar{q})+\parallel \exp^{-1}_{p^k}\bar{q}-\rm P_{\bar{p}p^k}\exp^{-1}_{\bar{p}}\bar{q}\parallel.
\end{equation*}
Taking limit with $k$ goes to infinity  in the last inequality and combining items (i) and (iii) we can conclude that $\lim_{k\to +\infty}  \exp^{-1}_{p^k}q^k=\exp^{-1}_{\bar{p}}\bar{q}$.
\end{proof}
In the  following result we present an important property of the parallel transport, which will be importante to prove our main result.
\begin{lemma}\label{Parallel transport}
Let $\bar{p}\in M$, $\bar{v}\in T_{\bar{p}}M$, $\bar{t}\in[0,1)$, $\{p^k\}\subset M$, $\{v^k\}\subset T_{p^k}M$  and $\{t_k\}\subset (0,1)$ be such that $\lim_{k\to +\infty} p^k=\bar{p}$, $\lim_{k\to +\infty}v^k=\bar{v}$ and $\lim_{k\to +\infty} t_k=\bar{t}$. Let $\{q^k\}$ be defined by  $q^k:=\exp_{p^k}t_kv^k$, for all  $k=0, 1, \ldots$.  Then,  there holds
\begin{equation*}
\lim_{k\rightarrow+\infty}\rm P_{p^kq^k}v^k = \rm P_{\bar{p}\bar{q}}\bar{v}.
\end{equation*}
where $\bar{q}:=  \exp_{\bar{p}}\bar{t}\bar{v}= \lim_{k\rightarrow+\infty}q^k$.
\end{lemma}
\begin{proof}
Since $M$ is a Hadamard manifold and $q^k=\exp_{p^k}t_kv^k$, for all  $k=0, 1, \ldots$,  we have
\begin{equation}\label{eq. Lemma1.3}
-\frac{1}{t_k}\exp^{-1}_{p^k}q^k = v^k, \qquad  k=0, 1, \ldots.
\end{equation}
Hence, the definition of   parallel transport along of the geodesic $\gamma_k(t)= \exp_{p^k}tv^k$  implies that
\begin{equation*}
\frac{1}{1-t_k}\exp^{-1}_{q^k}\exp_{p^k}v^k= -\frac{1}{t_k}\exp^{-1}_{q^k}p^k = P_{p^kq^k}v^k, \qquad  k=0, 1, \ldots.
\end{equation*}
Thus, taking the limit  as $k$ goes to $+\infty$ and using Lemma~\ref{eq:ContExp} we conclude that
\begin{equation*}
\lim_{k\rightarrow+\infty} P_{p^kz^k}v^k =\frac{1} {1-\bar{t}}\exp^{-1}_{\bar{q}}\exp_{\bar{p}}\bar{v}.
\end{equation*}
Since  $\bar{q}=  \lim_{k\rightarrow+\infty}q^k$, taking limit as   $k$ goes to $+\infty$in \eqref{eq. Lemma1.3} and combining with  the last equality we obtain the  desired result.
\end{proof}

Let $\Omega\subset M$ be a convex set, and $p\in \Omega$. Following \cite{LiLopesMartin-Marquez2009}, we define the {\it normal cone} to $\Omega$ at $p$ by
\begin{equation} \label{eq:nc}
N_{\Omega}(p):=\left\{w\in T_pM~:~\langle w, \exp_{p}^{-1}q \rangle\leq 0, ~q\in \Omega \right\}.
\end{equation}
Given a multivalued vector field  $X: M \rightrightarrows  TM$,   the  domain of $X$  is the set defined by
\begin{equation} \label{eq:dr}
\mbox{dom}X:=\left\{ p\in M ~: ~X(p)\neq \varnothing \right\},
\end{equation}
Let  $X: M \rightrightarrows  TM$ be a vector field and $\Omega\subset M$. We define the following quantity
 $$
 m_X(\Omega):=\sup_{q\in\Omega}\left\{\|u\|~:~u\in X(q)\right\}.
 $$
We say that  $X$  is {\it locally bounded} iff,  for all $p \in  \mbox{int}( \mbox{dom}\,X) $,   there exist an open  set $U\subset M$  such  that   $p\in U$ and  there holds $m_X(U) < +\infty,$ and {\it bounded on bounded sets }  iff for all bounded set $V\subset M$  such that its closure $\overline{V} \subset  \mbox{int}( \mbox{dom}\,X)$ it holds that   $m_X(V) < +\infty$; see an equivalent definition in \cite{LiLopesMartin-Marquez2009}. The multivalued vector field   $X$ is said to be {\it upper semicontinuous} at $p\in\mbox{dom}X $  iff, for any open set  $V\subset T_pM$ such that $X(p) \in V$, there exists an open set  $U \subset M$ with $p\in U$ such that $P_{qp}X(q)\subset V$, for any $q\in U$. For two  multivalued vector fields  $X, Y$ on $M$, the notation $ X\subset Y$ implies that $X(p)\subset Y(p)$, for all $p\in M$.
\begin{definition}
A sequence $\{p^k\} \subset (M,d)$ is said to Fej\'er convergence to a nonempty set $W\subset M$  if, for every $q \in W$,   we have $d^2(q,p^{k+1})\leq d^2(q,p^k)$, for $k=0,1, \ldots$ and    $ q\in W$.
\end{definition}
\begin{proposition}\label{fejer}
Let $\{p^k\}$ be a sequence in $(M,d)$. If $\{p^k\}$ is Fej\'er convergent to nonempty set $W\subset M$, then $\{p^k\}$ is bounded. If furthermore, an accumulation point $p$ of $\{p^k\}$ belongs to $W$, then $\lim_{k\rightarrow \infty}p^k=p$.
\end{proposition}
\begin{lemma}\label{lemmaseq.}
Let $\{\rho_k\}$ be a sequence of positive real numbers and $\theta_0>0$. Define the sequence $\{\theta_{k}\}$  by $\theta_{k+1}=\min\{\theta_k,\rho_k\}$. The limit $\bar{\theta}$ of $\{\theta_k\}$ is equal to 0 if and only if 0 is a cluster point of $\{\rho_k\}$.
\end{lemma}
\begin{proof}
See \cite[Lemma 4.9]{IusemPerez2000}.
\end{proof}
A multivalued vector field $X$ is said to be {\it monotone} iff
\begin{equation}\label{eq2.1}
\left\langle P_{pq} u-v, \, \exp_{q}^{-1}p\right\rangle \geq 0, \qquad  \qquad  ~ p,\,q\in \mbox{dom}X,  \quad u\in X(p), ~ v\in X(q).
\end{equation}
Moreover,  a monotone vector field $X$  is said to be  {\it maximal monotone\/}, iff for each $p\in \mbox{dom}\,X$ and $u\in T_pM$, there holds:
\begin{equation}\label{eq2.3}
\left\langle P_{pq} u-v, \,\exp_{q}^{-1}p\right\rangle \geq 0, \quad \qquad ~ q\in \mbox{dom}X, \quad ~ v\in X(q)  ~ \Rightarrow ~  u\in X(p).
\end{equation}
The concept of monotonicity in the Riemannian context was first  introduced in \cite{Nemeth1999},  for a single-valued case and   in \cite{NetoFerreiraLucambio2000} for a multivalued case. Further, the  notion of  maximal monotonicity of a vector field was  introduced  in \cite{LiLopesMartin-Marquez2009} and  in  \cite[Theorem~5.1]{LiLopesMartin-Marquez2009} is showed that  the  subdifferential  of the convex function  is  maximal monotone. The next  result  is an  extension to the Hadamard manifolds of its counterpart Euclidean and its  proof   is an immediate consequence of the definition of maximal monotonicity and normal cone, respectively.
\begin{lemma}\label{mon.cone}
Let $\Omega\subset M$ be a convex set and $X$ be a maximal monotone vector field such that $\emph{dom}\,X=M$. Then,  $X+N_{\Omega}$  is a maximal monotone vector field.
\end{lemma}
A multivalued vector field $X$ is said to be {\it upper Kuratowski semicontinuous} at $p$ if, for any open set $V$ satisfying $X(p)\subset V\subset T_pM$, there exists an open neighbourhood $U$ of $p$ such that $P_{qp}X(q)\subset V$ for any $q\in U$.  When $X$ is upper Kuratowski semicontinuous, for any $p\in \mbox{dom} X$, we say that $X$ is upper Kuratowski semicontinuous; see \cite{LiLopesMartin-Marquez2009}. It has been showed   in  \cite{LiLopesMartin-Marquez2009} that   maximal monotone vector fields are  upper semicontinuous Kuratowski and, if  additionally $\mbox{dom}X=M$, then $X$ is locally bounded $M$.

\section{ The enlargement of monotone vector fields} \label{sec3}
In this section,  we recall some properties of the $\epsilon$-enlargement  of maximal monotone vector fields and show  its  lower semicontinuous,  in order to obtain   the convergence of the extragradient method. We begin by recalling  the notion of   $\epsilon$-enlargement  for  multivalued  vector  fields  on Hadamard manifolds introduced in \cite{BatistaBentoFerreira2015-2}.

\begin{definition} \label{def.enl.X}
Let  $X$  be a multivalued monotone vector field on $M$   and $\epsilon \geq 0$.   The  $\epsilon$-enlargement   $X^{\epsilon}: M   \rightrightarrows  TM $  of $X$   is the  multivalued  vector  field  defined by
\begin{equation} \label{enl.X}
    X^{\epsilon}(p):=\left\{ u\in T_pM~:~ \left\langle P_{pq} u-v, \,\exp_{q}^{-1}p\right\rangle \geq  -\epsilon, ~  q\in \emph{dom}X, ~  v\in X(q) \right\}, \qquad p\in \emph{dom}X.
\end{equation}
\end{definition}
The  next  proposition is a combination of  \cite[Proposition 2.6, Proposition 2.2]{BatistaBentoFerreira2015-2}.
\begin{proposition} \label{prop.elem.ii}
Let  $X$ be a monotone vector field on $M$ and $\epsilon\geq0$. Then,  $\emph{dom}X \subset \emph{dom}X^\epsilon$. In particular,  if $\mbox{dom}X=M$ then $\emph{dom} X^\epsilon=\emph{dom}X$.  Moreover, if $X$ is maximal monotone then    $X^0=X$ and $X^\epsilon$ is bounded on bounded sets.
\end{proposition}
Throughout the paper we will also need of the following properties of  $\epsilon$-enlargement  of  maximal monotone vector fields. Its proofs can be  found in \cite{BatistaBentoFerreira2015-2}.
\begin{proposition} \label{prop.elem.X}
Let  $X$, $X_1$ and $X_2$  be  multivalued monotone vector fields on $M$ and  $\epsilon, \epsilon_1, \epsilon_2 \geq 0.$ Then, there hold:
\begin{itemize}
\item[i)] If $\epsilon_1\geq\epsilon_2\geq0$ then $X^{\epsilon_2}\subset X^{\epsilon_1}$;
\item[ii)] $X_1^{\epsilon_1}+X_2^{\epsilon_2}\subset(X_1+X_2)^{\epsilon_1+\epsilon_2}$;
\item[iii)] $X^\epsilon(p)$ is closed and convex for all $p\in M$;
\item[iv)] $\alpha X^\epsilon=(\alpha X)^{\alpha\epsilon}$ for all $\alpha\geq0$;
\item[v)] $\alpha X^\epsilon_1+(1-\alpha)X^\epsilon_2\subset(\alpha X_1+(1-\alpha) X_2)^\epsilon$ for all $\alpha\in[0,1]$;
\item[vi)] If $E\subset\mathbb{R}_+$, then $\bigcap_{\epsilon\in E}X^\epsilon=X^{\overline{\epsilon}}$ with $\overline{\epsilon}=\inf E$.
\end{itemize}
Moreover, if   $\{ \epsilon^k\}$ is a sequence of positive numbers and   $\{ (p^k, \, u^k)\} $ is a sequence in $TM$ such that
 $\overline{\epsilon}=\lim_{k\rightarrow\infty}\epsilon^k$, $\overline{p}=\lim_{k\rightarrow\infty}p^k$, $\overline{u}=\lim_{k\rightarrow\infty}u^k$ and $u^k\in X^{\epsilon_k}(p^k)$ for all $k$, then $\overline{u}\in X^{\overline{\epsilon}}(\overline{p})$.
\end{proposition}
\begin{proof}
See \cite[ Proposition 2.3]{BatistaBentoFerreira2015-2}.
\end{proof}
In the next definition we extend the notion of  lower semicontinuity of a multivalued operator, which has been introduced in \cite{BurachikIusemSvaiter1997},  to a vector field.
\begin{definition}
A multivalued vector field  $Y:M\rightrightarrows TM$ is said to be  lower semicontinuous at $\bar{p}\in \emph{dom}Y $ if,  for each sequence   $\{p^k\}\subset \emph{dom}Y$ such that  $\lim_{k\to + \infty}p^k=\bar{p}$ and each  $\bar{u}\in Y(\bar{p})$, there exists a sequence $\{w^k\}$ such that $w^k\in Y(p^k)$ and $\lim_{k\rightarrow\infty}P_{p^k\bar{p}}w^k=\bar{u}$.
\end{definition}
The next result is a generalization of  \cite[Theorem 4.1]{IusemPerez2000}, it will  play a  important rule  in the convergence analysis of the extragradient method.
\begin{theorem} \label{Theoremlsc}
Let $X:M\rightrightarrows TM$ be a maximal monotone vector field and $\epsilon>0$.  If  $\emph{dom}X=M$  then $X^{\epsilon}$ is lower semicontinuous.
\end{theorem}
\begin{proof}
Since $\mbox{dom}X=M$, thus   Proposition~\ref{prop.elem.ii} implies $\mbox{dom}X=\mbox{dom}X^{\epsilon}$. Take $\{p^k\}\subset M$ such that  $\lim_{k\to + \infty}p^k=\bar{p}$ and  $\bar{u}\in X^{\epsilon}(\bar{p})$.  First, we  prove that the following statements there hold:
\begin{enumerate}
\item[(i)] For each $0<\theta<1$ and $u^k\in X(p^k)$, there exists $k_0\in\mathbb{N}$ such that $(1-\theta)P_{\bar{p}p^k}\bar{u}+\theta u^k\in X^{\epsilon}(p^k)$, for all  $k>k_0$;
\item[(ii)] Take $\nu>0$.  Then, there exist $k_0\in\mathbb{N}$ and $v_k\in X^{\epsilon}(p^k)$ such that $\| \bar{u}-P_{p^k\bar{p}}v_k\|\leq\nu$, for all  $k>k_0$.
\end{enumerate}
For proving (i),  take $q\in M$ and $v\in X(q)$. Then,  simple algebraic manipulations yield
\begin{multline*}
 \left\langle P_{p^kq}[(1-\theta)P_{\bar{p}p^k}\bar{u}+\theta u^k]-v, ~ \exp^{-1}_q{p^k}\right\rangle=\\(1-\theta) \left\langle P_{\bar{p}q}\bar{u}-v,\exp^{-1}_{q}p^k\right\rangle+\theta\left\langle P_{p^kq}u^k-v,\exp^{-1}_{q}p^k\right\rangle.
\end{multline*}
Since $X$ is monotone, the second term in the right hand said of the last inequality is positive. Thus,
\begin{equation}\label{lemma_lsc_i}
\left\langle P_{p^kq}[(1-\theta)P_{\bar{p}p^k}\bar{u}+\theta u^k]-v, ~ \exp^{-1}_q{p^k}\right\rangle \geq (1-\theta) \left\langle P_{\bar{p}q}\bar{u}-v,\exp^{-1}_{q}p^k\right\rangle.
\end{equation}
Considering that $\bar{u}\in X^{\epsilon}(\bar{p})$ and  $\lim_{k\rightarrow\infty} \left\langle P_{\bar{p}q}\bar{u}-v,\exp^{-1}_{q}p^k\right\rangle= \left\langle P_{\bar{p}q}\bar{u}-v,\exp^{-1}_{q}\bar{p}\right\rangle\geq-\epsilon$  then,  for all $\delta>0$, there exists $k_0\in\mathbb{N}$ such that
\begin{equation}\label{lemma_lsc_ii}
 \left\langle P_{\bar{p}q}\bar{u}-v,\exp^{-1}_{q}p^k\right\rangle\geq-\epsilon-\delta,  \qquad  k >k_0.
\end{equation}
Combining (\ref{lemma_lsc_i}) and (\ref{lemma_lsc_ii}) and taking $\delta=\theta\epsilon/(1-\theta)$  we  conclude that
\begin{eqnarray*}
\left\langle P_{p^kq}[(1-\theta)P_{\bar{p}p^k}\bar{u}+\theta u^k]-v, ~ \exp^{-1}_q{p^k}\right\rangle \geq-\epsilon,   \qquad  k >k_0,
\end{eqnarray*}
which proof the item (i).  For proving the item (ii),  take $\eta>0$ and consider  the following   constants:
$$
 \sigma:=\sup\left\{\parallel u\parallel: u\in X^{\epsilon}(B(\bar{p},\eta))\right\}, \qquad \gamma:=\mbox{min}\{(\epsilon/2\sigma),\eta\}, \qquad 0<\mu <\mbox{min}\{1,(\nu/2\sigma)\}.
$$
Take any $u^k\in X(p^k)$.   Applying  item (i)  with $\theta=\mu$, we conclude that there exists $k_0\in \mathbb{N}$ such that $(1-\mu)P_{\bar{p}p^k}\bar{u}+\mu u^k\in X^{\epsilon}(p^k)$, for all $k\geq k_0$. We shall to  prove that,  taking $v_k=(1-\mu)P_{\bar{p}p^k}\bar{u}+\mu u^k$,  we have $\parallel\| \bar{u}-P_{p^k\bar{p}}v_k\|\parallel\leq \nu$,  for all  $k\geq k_0$. First note that,  some manipulation and taking into account that  the parallel transport is an isometry we have
\begin{equation}\label{est. bar{u}-u^k}
\| \bar{u}-P_{p^k\bar{p}}v_k\|=\mu\parallel \bar{u}-P_{{p^k}\bar{p}}u^k\parallel\leq\mu(\parallel \bar{u}\parallel+\parallel u^k\parallel), \qquad k\geq k_0.
\end{equation}
Since $\lim_{k\to + \infty}p^k=\bar{p}$, there exist   $k_0$ such that $p^k\in B(\bar{p},\gamma)$,  for all $k\geq k_0$.  Thus, taking into account that
$u^k\in X(p^k)\subset X^{\epsilon}(p^k)$ and $B(\bar{p},\gamma)\subset B(\bar{p},\eta)$, the definition of $\sigma$  gives  $\parallel u^k\parallel\leq\sigma$,  for all $k\geq k_0$.  Due to   $\bar{u}\in X^{\epsilon}(\bar{p})$ we also have  $\parallel \bar{u}\parallel\leq\sigma.$ Therefore, using \eqref{est. bar{u}-u^k} and the definition of $\mu$ we obtain
\begin{equation}
\parallel \bar{u}-v_k\parallel\leq2\sigma\mu\leq\nu, \qquad  k\geq k_0,
\end{equation}
and the proof of item (ii) is proved. Finally,  we define the sequence $\{w^k\}$  as follows $w^k:=\argmin \{\|\bar{u}-P_{p^k\bar{p}}u \| ~ : ~ u\in X^{\epsilon}(p^k)\}$  for each $k$.  Since, for each $k$ the set $X^{\epsilon}(p^k)$ is closed and convex, the sequence $\{w^k\}$ is well defined.  We claim that $\lim_{k\rightarrow\infty}P_{p^k\bar{p}}w^k=\bar{u}$. Otherwise there exists  $\{p^{k_j}\}$,   a subsequence of $\{p^k\}$,  and some $\nu>0$ such that $\parallel \bar{u}-P_{p^{k_j}\bar{p}}w^{k_j}\parallel>\nu$ for all $j$.  Definition of the sequence  $\{w^k\}$ implies that  $\parallel \bar{u}- P_{p^{k_j}\bar{p}}u\parallel>\nu$ for all $u\in X^{\epsilon}(p^{k_j})$ and all $j$.  On the other hand,  considering  that  $\lim_{k_j\to + \infty}p^{k_j}=\bar{p}$,  $\bar{u}\in X^{\epsilon}(\bar{p})$  and the item (ii) holds,  for all $\nu>0$, we have a contraction. Therefore,  the claim is proven and  the proof is concluded.
\end{proof}
\begin{remark}
 It is well known that the  lower semicontinuity of operators  is an  important property which is useful to obtain convergence results of iterative process in the  Euclidean space. We also remark that, in general, maximal monotone operator  are not  lower semicontinuous; see \cite[Section~2]{IusemPerez2000}. The result in the above theorem establishes  this property to vector fields and  it will be used,  in the next section,   to prove the convergence of the extragradient-type algorithm for variational inequality in Riemannian manifolds.
\end{remark}
\section{ An extragradient-type algorithm for variational inequality} \label{sec4}
In this section,  we introduce  an  extragradient-type algorithm for variational inequalities problem in Hadamard manifolds. It is worth pointing out that,  the variational inequality problem  was first  introduced in \cite{Nemeth2003},  for  single-valued  vector  fields  on Hadamard manifolds,   and   in \cite{LiYao2012} for  multivalued  vector  fields in  general Riemannian manifolds.   Let $X:M\rightrightarrows TM$ be a multivalued vector field and $ \Omega \subset M$ be a nonempty set.  The {\it variational inequality problem} $\mbox{VIP}(X,\Omega)$ consists of finding $p^*\in\Omega$ such that there exists $u^*\in X(p^*)$ satisfying
$$
\langle u^*,\,\exp^{-1}_{p^*}q\rangle\geq0, \qquad   q\in\Omega.
$$
Denote by $S^*(X, \Omega)$ {\it  the solution} set of $\mbox{VIP}(X,\Omega)$. From  the definition of normal  cone, we  can show  that  $\mbox{VIP}(X,\Omega)$  becomes the problem of finding $p^*\in\Omega$ such that
\begin{equation} \label{eq.vip}
0\in X(p^*)+N_{\Omega}(p^*).
\end{equation}
Note that,   if  $\Omega=M$ then  $N_{\Omega}(p)=\{0\}$. Hence,  VIP(X,$\Omega$) becomes   the problem of   finding $p^*\in\Omega$ such that $0\in X(p^*).$ This particular instance has been studied in several  papers, including \cite{FerreiraPerezNemeth2005, LiLopesMartin-Marquez2009}.
\begin{lemma} \label{solution_condition}
The following statements are equivalent:
\begin{enumerate}
\item[i)] $p^*$ is a solution of $\emph{VIP}(X,\Omega)$;
\item[ii)] There exists $u^* \in X(p^*)$ such that $p^*=P_{\Omega}(\exp_{p^*}(-\alpha u^*))$,  for some $\alpha>0$.
\end{enumerate}
\end{lemma}
\begin{proof}
It is an immediate consequence of the inequality   \eqref{eq:proj}.
\end{proof}
\noindent
To analyze the extragradient  algorithm   we need  of the following   three assumptions:
\begin{itemize}
\item[{\bf A1.}] $\dom X=M$ and $\Omega\subset M$ is closed and convex;
\item[{\bf A2.}] $X$ is maximal monotone;
\item[{\bf A3.}] $S^*(X, \Omega) \neq\varnothing$.
\end{itemize}
We also need of the following   assumption, which  plays an important rule in the convergence  analyses of the our extragradient  algorithm in Hadamard manifolds.
\begin{itemize}
\item[{\bf A4.}] For each $y\in M$ and   $v\in T_yM$   the   following set  is convex
\begin{equation}\label{def.S}
S:=\left\{x\in M:\langle v,\exp^{-1}_yx\rangle\leq0\right\}.
\end{equation}
\end{itemize}
\begin{remark}
If the manifold  $M$ is the Euclidean space then    $S$ in  \eqref{def.S}  is a closed  semi-space.  In  \cite{FerreiraPerezNemeth2005} has been shown that for Hadamard manifolds with constant  curvature or two dimensional  the set $S$  is convex. However,  so far  it not known if  $S$ is or not convex  in  general Hadamard manifolds.
\end{remark}
Next, we present  an{ \it  extragradient-type algorithm  for finding a solution of  $\mbox{VIP}(X,\Omega)$}.

\noindent

\hrule
\noindent
\\
{\bf  Extragradient-type algorithm\\} \label{EA}
\hrule
\vspace{1mm}
\noindent
Our algorithm  requires six exogenous constants:
$$
\epsilon>0,\qquad  0<\delta_-<\delta_+<1,\qquad  0<\alpha_-<\alpha_+,\qquad  0<\beta<1,
$$
and two exogenous sequences $\{\alpha_k\}$ and $\{\beta_k\}$ satisfying the fowling conditions:
$$
\alpha_k\in[\alpha^-, ~\alpha^+],\qquad  \beta_k\in[\beta,~1], \qquad k=0,1, \ldots.
$$
{\bf 1.} {\sc Initialization: } $p^0\in \Omega$, \qquad $\epsilon_0=\epsilon$.\\
{\bf 2.} {\sc  Iterative step: }Given $p^k$ and $\epsilon_k$,
\begin{itemize}
\item[\bf{(a)}] \emph{ Selection of $u^k$: } Find
\begin{equation}\label{sel.uk.i}
u^k\in X^{\epsilon_k}(p^k),
\end{equation}
such that
\begin{equation}\label{sel.uk.ii}
\left\langle w,-\exp^{-1}_{p^k}P_{\Omega}(\exp_{p^k}(-\alpha_ku^k)) \right\rangle \geq \frac{\delta_+}{\alpha_k}d^2\left(p^{k}, P_{\Omega}(\exp_{p^k}(-\alpha_ku^k))\right),  \quad  w\in X^{\epsilon_k}(p^k).
\end{equation}
Define,
\begin{equation}\label{sel.uk.iii}
z^k:=P_{\Omega}\left(\exp_{p^k}(-\alpha_ku^k)\right).
\end{equation}
\item[\bf{(b)}]  \emph{Stopping criterion:} If $p^k=z^k$ then stop. Otherwise,
\item[\bf{(c)}]  \emph{ Selection of $\lambda_k$ and $v^k$: } Define $\gamma_k(t):=\exp_{p^k}t\exp^{-1}_{p^k}z^k$ and let
\begin{equation}\label{ik}
i(k):=\mbox{min}\left\{i\geq0~:~  \exists \; v^{k,i}\in X(y^{k,i}), ~\left\langle v^{k,i},\gamma'_k(2^{-i}\beta_k) \right\rangle \leq -\frac{\delta_-}{\alpha_k}d^2(p^{k},z^{k})\right\},
\end{equation}
where \begin{equation}\label{yki}
y^{k, i}:=\gamma_k\left(2^{-i}\beta_k\right).
\end{equation}
Define
\begin{align}\label{lambda_k}
 y^{k}&:=\exp_{p^k} \lambda_k \exp^{-1}_{p^k}z^k,   \qquad  \qquad  \lambda_k:=2^{-i(k)}\beta_k, \\
\quad v^k &:= v^{k,i(k)}.  \label{v^k}
\end{align}
\item[\bf{(d)}] \emph{ Definition of $p^{k+1}$ and  $\epsilon_{k+1}$:} Define
\begin{align}\label{S_k}
q^k &:=P_{S_k}(p^k),  \qquad \qquad S_k:=\{p\in M:\langle v^k,\exp^{-1}_{y^k}p\rangle\leq0\},\\
p^{k+1} & :=P_{\Omega}(q^k), \qquad  \qquad  \epsilon_{k+1}:=\mbox{min} \left\{\epsilon_{k}, ~d^2(p^{k},z^{k})\right\},  \label{def.p_k}
\end{align}
set $k\leftarrow k+1$ and go to {\sc  Iterative step }.
\end{itemize}
\hrule
\noindent
\vspace{0.005cm}
\begin{remark}
Assuming that $X$ is an univalued monotone vector field and $\epsilon_k=0$,  for all $k$,  the previous algorithm becomes the algorithm presented in \cite{TangHuang2012} (see also, \cite{FerreiraPerezNemeth2005})  and moreover, in the  particular case where  $M=\mathbb{R}^n$,   it merges  into the extragrdient algorithm presented in  \cite{IusemPerez2000}.
\end{remark}
Next result establishes  the well-definedness of the above   algorithm.
\begin{lemma}\label{Boa_def.ii}
The following statements  hold:
\begin{enumerate}
\item[\emph{(i)}] $p^k\in\Omega$, for all $k$;
\item[\emph{(ii)}] there exists $u^k$ satisfying (\ref{sel.uk.i}) and  (\ref{sel.uk.ii}),  for each $k$;
\item[\emph{(iii)}] if $p^k\neq z^k$,  then $i(k)$ is well defined.
\end{enumerate}
\end{lemma}
\begin{proof} The proof of  item $(i)$ follows  from the initialization step and    \eqref{def.p_k}.
\noindent
For proving  item $(ii)$ define,  for each $k$,  the bifunction   $f_k:T_pM\times T_pM\rightarrow \mathbb{R}$ by
\begin{equation} \label{eq:afim}
f_k(u,v)=\left\langle \exp^{-1}_{p^k}z^k,~u-v\right\rangle.
\end{equation}
 In view of  Proposition  \ref{prop.elem.ii} and  item iii of  Proposition \ref{prop.elem.X}  we have $X^{\epsilon_k}(p^k)$ compact and convex.  Hence,  applying    \cite[Basic Existence Theorem on page 3]{GiancarloCastellaniPappalardoPassacantando2013} (see also \cite[Theorem~3.1]{BatistaBentoFerreira2015}) with $C=X^{\epsilon_k}(p^k)$ and $f=f_k$,  we conclude that there exists  $u^k\in X^{\epsilon_k}(p^k)$ such that $f_k(u^k,v)\geq 0$ for all $v\in  X^{\epsilon_k}(p^k)$,  and from  \eqref{eq:afim} we obtain
\begin{equation}\label{e.p.}
\langle \exp^{-1}_{p^k}z^k,u^k\rangle\geq\langle \exp^{-1}_{p^k}z^k, v\rangle, \qquad v\in  X^{\epsilon_k}(p^k).
\end{equation}
On the other hand,   using  inequality    \eqref{eq:coslaw2} with $p_1=q$, $p_2=p^k$ and $p_3=z^k$  we have
$$
\langle \exp^{-1}_{p^k}q, \, \exp^{-1}_{p^k}{z^k}\rangle+\langle \exp^{-1}_{z^k}q, \, \exp^{-1}_{z^k}p^k\rangle \geq d^2(p^k, z^k),
$$
where $q=\exp_{p^k}(-\alpha_ku^k/\delta_+ )$.  Since  $z^k=P_{\Omega}(q)$,  it follows  from  \eqref{eq:proj} and the  last inequality that
$$
\langle \exp^{-1}_{p^k}q,\exp^{-1}_{p^k}z^k\rangle\geq d^2(p^k,z^k).
$$
Thus, considering that $q=\exp_{p^k}(-\alpha_ku^k/\delta_+ )$, we have $\langle u^k,\exp^{-1}_{p^k}z^k\rangle\leq - \delta_+d^2(p^k,z^k)/\alpha_k$ which,  combined with \eqref{e.p.},  gives the desired result.

\noindent
To prove  item (iii)  we proceed by contradiction.  Fix $k$  and  assume that, for each $i$, there holds
\begin{equation}\label{itemiii}
\left\langle v^{k,i},~\gamma_k'(2^{-i}\beta_k)\right\rangle>-\frac{\delta_-}{\alpha_k}d^2(p^k,z^k), \qquad v^{k,i} \in  X(y^{k,i}).
\end{equation}
First note that, from \eqref{yki}  we have    $\{y^{k,i}\}$ belongs to the geodesic segment joining $p_k$ to $\gamma_{k}(\beta_k)$. Thus   $\{y^{k,i}\}$ is a bounded sequence and, consequently,     \cite[Proposition~3.5]{LiLopesMartin-Marquez2009} implies that  $\{X(y^{k,i})\}$ is also a bounded. Considering that $v^{k,i}\in X(y^{k,i})$  for each $i$,   without loss of generality, we can  assume that  $\{v^{k,i}\}$ converges to  $\bar{v}$. Letting $i$ goes to infinity in \eqref{itemiii} and taking into account that   $\lim_{i\to \infty}v^{k,i}=\bar{v}$, $\gamma(0)=p^k$ and $\gamma'(0)=\exp^{-1}_{p^k}z^k$,  we conclude that
\begin{equation}\label{itemiii.2}
\left\langle \bar{v}, ~\exp^{-1}_{p^k}z^k\right\rangle\geq-\frac{\delta_-}{\alpha_k}d^2(p^k,z^k).
\end{equation}
On the other hand, since     $\lim_{i\rightarrow\infty}y^{k,i}=p^k$, $\lim_{i\to \infty}v^{k,i}=\bar{v}$ and  $v^{k,i} \in  X(y^{k,i})$  for each $i$,  using     \cite[Proposition~3.5]{LiLopesMartin-Marquez2009} we have $\bar{v}\in X(p^k)$. Thus,  combining Propositions \ref{prop.elem.ii} and \ref{prop.elem.X} (i) we obtain  $\bar{v}\in X^{\epsilon}(p^k)$. Hence,  using  \eqref{sel.uk.iii} and  taking  $w=\bar{v}$ the inequality  \eqref{sel.uk.ii} becomes
\begin{equation}\label{itemiii.3}
\langle \bar{v},\exp^{-1}_{p^k}z^k)\rangle\leq-\frac{\delta_+}{\alpha_k}d^2(p^k,z^k).
\end{equation}
Since $0<\delta^{-}<\delta^{+}$,  the inequalities  \eqref{itemiii.2} and \eqref{itemiii.3} imply that $d(p^k,z^k)=0$, which is a contradiction with  $p^k\neq z^k$. Therefore, $i(k)$ is well defined and the proof of the proposition is done.
\end{proof}
From now on, $\{p^k\}$, $\{q^k\}$, $\{y^k\}$, $\{z^k\}$,  $\{v^k\}$, $\{u^k\}$ and $\{\epsilon_k\}$ denote sequences generated by extragradient-type algorithm.  To prove the convergence of  $\{p^k\}$  to  a point of the solution set  $S^*(X, \Omega)$,   we need  some preliminaries  results.
\begin{lemma}\label{Fejer}
The sequence  $\{p^k\}$ is F\'ejer convergent to $S^*(X, \Omega)$ and $\lim_{k\rightarrow\infty}d(q^k,p^k)=0$.
\end{lemma}
\begin{proof}
First let us  show that, for all   $ p^*\in S^*(X, \Omega)$  there holds
\begin{equation}\label{eq.Fejer}
d^2(p^*,p^{k+1})\leq d^2(p^*,p^k)-d^2(q^k,p^k), \qquad k=0,1, \ldots.
\end{equation}
For that, take     $u^*\in X(p^*)$  and fix $k$.  Hence, due the monotonicity of   $X$,   we conclude that
$$
\langle v^k,\exp^{-1}_{y^k}p^*\rangle\leq 0.
$$
Thus, in view of \eqref{S_k}, we obtain that  $p^*\in S_k$. On the other hand, applying \eqref{eq:coslaw} with $p_1=p^*$, $p_2=p_k$ and $p_3=q_k$ we have
\begin{eqnarray*}
d^2(p^*,p^k)\geq d^2(p^*,q^k)+d^2(q^k,p^k)-2\left\langle \exp^{-1}_{q^k}p^*,\exp^{-1}_{q^k}p^k\right\rangle.
\end{eqnarray*}
Since $p^*\in S_k$ and $q^k=P_{S_k}(p^k)$, the last inequality implies that
\begin{eqnarray*}
d^2(p^*,p^k)\geq d^2(p^*,q^k)+d^2(q^k,p^k).
\end{eqnarray*}
Analogously,  applying \eqref{eq:coslaw} with $p_1=p^*$, $p_2=q_k$ and $p_3=p_{k+1}$ and considering that  $p^{k+1}:=P_{\Omega}(q^k)$ and $ p^*\in \Omega$,   we conclude that
\begin{eqnarray*}
d^2(p^*,q^k)\geq d^2(p^*,p^{k+1})+d^2(q^k,p^{k+1}).
\end{eqnarray*}
Combining the two last inequalities,  we obtain
$
d^2(p^*,p^k)\geq d^2(q^k,p^k)+d^2(p^*,p^{k+1})+d^2(q^k,p^{k+1}),
$
which implies  \eqref{eq.Fejer}. In particular,  \eqref{eq.Fejer} implies   that $\{p^k\}$ is F\'ejer convergent to $S^*(X, \Omega)$ and $\{d(p^*,p^k)\}$ is noincresing.  Therefore, owing  $\{d(p^*,p^k)\}$ is   inferiorly limited we conclude that    $\{d(p^*,p^k)\}$ converges and taking into account    \eqref{eq.Fejer} the desired result follows.
\end{proof}
\begin{lemma}\label{convergence}
If the sequence $\{p^k\}$ is infinity then $\lim_{k\rightarrow\infty}\epsilon_k=0$. Moreover,   all   cluster points  of $\{p^k\}$  belong to $S^*(X, \Omega)$.
\end{lemma}
\begin{proof} Suppose that the sequence $\{p^k\}$ is infinity, i.e., the algorithm does not stop. Thus,  by the stopping criterion $d(p^k,z^k)>0$,  for all $k$.  Since   \eqref{def.p_k} implies that $\{\epsilon_k\}$ is a  nonincreasing  monotone sequence and  being   nonnegative   it follows that it converges. Set $\bar{\epsilon}:=\lim_{k\to +\infty}\epsilon_k$. We are going to  prove that $\bar{\epsilon}=0$. First of all, note that from  Lemma \ref{Fejer}  the sequence $\{p^k\}$ is Fej\'er convergent to $S^*(X, \Omega) $ and,  due to   {\bf A3},   we have $S^*(X, \Omega) \neq\varnothing$. Hence,  we conclude that $\{p^k\}$ is bounded.  On the other hand,  considering that $\{p^k\}$ is bounded, Proposition~\ref{prop.elem.ii} implies that $X^{\epsilon_0}(\{p^k\})=\cup_{k=0}^{\infty}X^{\epsilon_0} (p^k)$ is  bounded.    Since $\epsilon_k\leq \epsilon_0$,  the  item $i)$ of Proposition~\ref{prop.elem.X} implies that $X^{\epsilon_k}\subset X^{\epsilon_0} $, for all $k$, and  from  \eqref{sel.uk.i} we obtain that  $\{u^k\}$ is  bounded.  Definitions of $\lambda_k$ and $y^k$ in \eqref{lambda_k}  implies that  $y^k$  belongs to the geodesic  joining $ p^k$ to $z^k$ and,  using \eqref{eq:projnon-expansive} and  \eqref{sel.uk.iii},  we have
$$
d(p^k, y^k)\leq d(p^k, z^k)=d\left(P_{\Omega}(p^k), P_{\Omega}\left(\exp_{p^k}(-\alpha_ku^k)\right)\right)\leq d(p^k, \exp_{p^k}(-\alpha_ku^k))=\|\alpha_ku^k\|,
$$
for $k=0, 1, \ldots.$ In view of the  boundedness of the sequences  $\{p^k\}$,  $\{u^k\}$  and  $\{\alpha_k\}$, we obtain  from the last inequalities  that  $\{y^k\}$ and   $\{z^k\}$  are   bounded. Considering that $v^k\in X(y^k)$,  for all k, we apply  Proposition~\ref{prop.elem.ii}  to  conclude that $\{v^k\}$ is bounded.   Note that \eqref{ik}, \eqref{yki}, \eqref{lambda_k} and \eqref{v^k} imply
\begin{eqnarray*}
\left\langle v^k,\gamma'_k(\lambda_k)\right\rangle\leq-\frac{\delta_{-}}{\alpha_k}d^2(p^k,z^k), \qquad k=0, 1, \ldots.
\end{eqnarray*}
Combining \eqref{ik}, \eqref{yki} and \eqref{lambda_k}, it follows  that $\gamma'(\lambda_k)=-\lambda_k^{-1}\exp^{-1}_{y^k}p^k$, for $ k=0, 1, \ldots$. Thus, taking into account that $ 0< \alpha_k< \alpha_+$, last inequality becomes
\begin{equation} \label{eq.conv.i}
\left\langle v^k,\exp^{-1}_{y^k}p^k\right\rangle\geq\frac{\lambda_k\delta_{-}}{\alpha_+}d^2(p^k,z^k),  \qquad k=0, 1, \ldots .
\end{equation}
Since $\{p^{k}\}$, $\{u^{k}\}$, $\{v^k\}$, $\{z^{k}\}$,  $\{y^{k}\}$,  $\{\alpha_{k}\}$, and $\{\lambda_{k}\}$ are bounded, without loss of generality, we can assume that  they have   subsequences $\{p^{k_j}\}$, $\{u^{k_j}\}$, $\{v^{k_j})\}$, $\{z^{k_j}\}$, $\{y^{k_j}\}$, $\{\alpha_{k_j}\}$ and $\{\lambda_{k_j}\}$ converging to   $\bar{p}$, $\bar{u}$, $\bar{v}$, $\bar{z}$, $\bar{y}$,  $\bar{\alpha}$ and $\bar{\lambda}$,  respectively. Note that,  \eqref{S_k} yields
$$
q^{k_j}\in S_{k_j}=\left\{p\in M~:~  \left\langle v^{k_j}, \exp^{-1}_{y^{k_j}}p \right \rangle\leq 0\right\}, \qquad j=0,1, \ldots.
$$
Using Lemma \ref{Fejer} we have $\lim_{j\rightarrow\infty}p^{k_j}=\lim_{j\rightarrow\infty}q^{k_j}=\bar{p}$. Thereby, latter inclusion  together with item (iv) of  Lemma \ref{eq:ContExp} and $\lim_{j\rightarrow\infty}y^{k_j}=\bar{y}$  implies
$$
\lim_{j\rightarrow\infty} \left\langle v^{k_j}, \exp^{-1}_{y^{k_j}}p^{k_j}\right\rangle=\lim_{j\rightarrow\infty}\left\langle v^{k_j}, \exp^{-1}_{y^{k_j}}q^{k_j}\right\rangle\leq0.
$$
Thus  combining the last inequality with   (\ref{eq.conv.i}) it follows  that
\begin{equation} \label{Boa_def.iii}
\lim_{j\rightarrow\infty}\lambda_{k_j}d^2(p^{k_j},z^{k_j})=0.
\end{equation}
Considering that  $\lim_{j\rightarrow\infty}\lambda_{k_j}=\bar \lambda$,  we have two possibilities: either $\bar \lambda>0$ or $\bar \lambda=0$ . First, let us to assume that $\bar \lambda>0$. Since $\lim_{ j\to \infty} p^{k_j}= \bar{p}$ and $\lim_{j \to \infty} z^{k_j}= \bar{z}$, we obtain  from (\ref{Boa_def.iii}) that
$
d(\bar{p},\bar{z})=\lim_{j\to \infty}d(p^{k_j},z^{k_j})=0,
$
and consequently $\bar{p}=\bar{z}$.  Moreover, $0$  is a cluster point of $\{ d^2(p^{k},z^{k})\}$.  Hence, taking into account  \eqref{def.p_k}, we can apply  Lemma~\ref{lemmaseq.} with $\theta_k=\epsilon_k$ and $\rho_k=d^2(p^k,z^k)$  to conclude that $0=\lim_{k\to +\infty}\epsilon_k=\bar{\epsilon}$. Owing to $u^{k_j}\in X^{\epsilon_{k_j}}(p^{k_j})$, combining  Propositions~\ref{prop.elem.X} and \ref{prop.elem.ii}  we conclude that   $\bar{u}\in X(\bar{p})$.
Therefore,  Lemma \ref{solution_condition} implies that $\bar{p}\in S^*(X, \Omega)$.  Now, let us to assume that $\bar \lambda=0$.   In this case, using Lemma~\ref{eq:ContExp} and \eqref{yki} we conclude that $\lim_{ j\to \infty} y^{k_j,i(k_j)-1}=\bar p$. From  Proposition~\ref{prop.elem.ii}  we can take a sequence $\{\xi^j\}$ such that $\xi^j\in X(y^{k_j,i(k_j)-1})$ with  $\lim_{ j\to \infty}\xi^j=\bar{\xi}$ and,  by using  \cite[Proposition~3.5]{LiLopesMartin-Marquez2009},   we conclude that  $\bar{\xi}\in X(\bar{p})$. On the other hand,  \eqref{ik} implies
$$
-\left\langle \xi^j,\gamma'_{k_j}(2^{-i(k_j)+1}\beta_{k_j})\right\rangle<\frac{\delta^{-}}{\alpha_{k_j}}d^2(p^{k_j},z^{k_j}), \quad j=0, 1, \dots.
$$
Considering that  $\gamma'_{k_j}(2^{-i(k_j)+1}\beta_{k_j})= P_{p^{k_j} y^{k_j,i(k_j)-1}}\exp^{-1}_{p^{k_j}}z^{k_j}$ and the parallel transport is an isometry the last inequality  yields
$$
-\left\langle \xi^j,  P_{p^{k_j} y^{k_j,i(k_j)-1}}\exp^{-1}_{p^{k_j}}z^{k_j}\right\rangle<\frac{\delta^{-}}{\alpha_{k_j}}d^2(p^{k_j},z^{k_j}), \quad j=0, 1, \dots.
$$
Taking limits  in the above inequality,  as $j$ goes to infinity, and using item~$(iv)$ of Lemma~\ref{eq:ContExp} together  Lemma~\ref{Parallel transport} we obtain
\begin{equation} \label{eq:inqml1}
-\langle \bar{\xi},\exp^{-1}_{\bar{p}}\bar{z}\rangle\leq\frac{\delta^{-}}{\bar{\alpha}}d^2(\bar{p},\bar{z}).
\end{equation}
Assume by contradiction that $\bar{\epsilon}>0$. First note that Theorem \ref{Theoremlsc} implies that  $X^{\bar{\epsilon}}$ is lower semicontinuous. Therefore,  due to   $\lim _{j\rightarrow\infty}p^{k_j}=\bar{p}$ and $\bar{\xi}\in X(\bar{p})\subset X^{\bar{\epsilon}}(\bar{p})$,   there exists a sequence  $\{P_{p^{k_j}\bar{p}}w^j\}$ with $w^j\in X^{\bar{\epsilon}}(p^{k_j})$  such that  $\lim _{j\rightarrow\infty}P_{p^{k_j}\bar{p}}w^j=\bar{\xi}$. Besides,  \eqref{def.p_k}  implies that   $\bar{\epsilon}\leq\epsilon_{k_j}$ and,  using item i of Proposition~\ref{prop.elem.X},  we conclude that $X^{\bar{\epsilon}}(p^{k_j})\subset X^{\epsilon_{k_j}}(p^{k_j})$, for all $j$. Henceforth,  $w^j\in X^{\epsilon_{k_j}}(p^{k_j})$,  for all $j$, and from  (\ref{sel.uk.ii}) we have
$$
\langle w^j,-\exp^{-1}_{p^{k_j}}z^{k_j}\rangle\geq\frac{\delta^{+}}{\alpha_{k_j}}d^2(p^{k_j},z^{k_j}), \quad j=0, 1, \dots.
$$
Letting  $j$ goes to infinity  in the last inequality and considering Lemma~\ref{eq:ContExp} we obtain that
$$
-\langle \bar{\xi},\exp^{-1}_{\bar{p}}\bar{z}\rangle\geq\frac{\delta^{+}}{\bar{\alpha}}d^2(\bar{p},\bar{z}).
$$
Since $\bar{\alpha}\geq\alpha^->0$ and $0<\delta^-<\delta^+$, combining last inequality with \eqref{eq:inqml1}  we conclude that    $\bar{p}=\bar{z}$. Again, taking into account  \eqref{def.p_k}, we can apply  Lemma~\ref{lemmaseq.} with $\theta_k=\epsilon_k$ and $\rho_k=d^2(p^k,z^k)$  to conclude that $0=\lim_{k\to +\infty}\epsilon_k=\bar{\epsilon}$, which is  a contradiction. Due to $u^{k_j}\in X^{\epsilon_{k_j}}(p^{k_j})$, combining  Propositions \ref{prop.elem.X} and \ref{prop.elem.ii}  we conclude that   $\bar{u}\in X(\bar{p})$. Hence,  Lemma \ref{solution_condition} implies that $\bar{p}\in S^*(X, \Omega)$. Therefore, the proof is concluded.
\end{proof}
\begin{theorem} \label{eq:Conv}
 Either the sequence $\{p^k\}$ is finite and ends at iteration k, in which case $p^k$ is $\epsilon_k$-solution of $\mbox{VIP}(X,\Omega)$, i.e.,
\begin{equation} \label{eq:epssolu}
\sup_{q\in\Omega, v\in X(q)}\left\langle v,\exp^{-1}_q{p^k}\right\rangle \leq   \epsilon_k,
\end{equation}
or it is infinite, in which case it converges to a solution of $\mbox{VIP}(X,\Omega)$.
\end{theorem}
\begin{proof}
If the algorithm stops at the iteration $k$, then  we have  $p^k=z^k=P_{\Omega}(\exp_{p^k}(-\alpha_ku^k))$. Since  $u^k\in X^{\epsilon_k}(p^k)$,  Definition~\ref{def.enl.X} implies
$$
-\left\langle u^k,\exp^{-1}_{p^k}q\right\rangle-\left\langle v,\exp^{-1}_q{p^k}\right\rangle\geq-\epsilon_k, \qquad q\in\Omega, \quad v\in X(q).
$$
Taking into account  that  $\alpha_k>0$ and $p^k=z^k$,   the last inequality can be written as
$$
\frac{1}{\alpha_k}\left\langle \exp^{-1}_{z^k}[\exp_{p^k}(-\alpha_ku^k)],\exp^{-1}_{z^k}q\right\rangle-\left\langle v,\exp^{-1}_q{p^k}\right\rangle\geq-\epsilon_k, \qquad q\in\Omega, \quad v\in X(q).
$$
In view of \eqref{eq:proj}  and considering that $z^k=P_{\Omega}(\exp_{p^k}(-\alpha_ku^k))$ we conclude  from last inequality that
$$
\left\langle v,\exp^{-1}_q{p^k}\right\rangle\leq\epsilon_k,  \qquad q\in\Omega, \quad v\in X(q),
$$
which implies the desired inequality. Therefore, $p^k$ is an $\epsilon_k$-solution of VIP$(X,\Omega).$ Now,  if $\{p^k\}$ is infinite, then from Lemma \ref{Fejer} the sequence  $\{p^k\}$ is F\'ejer convergent to $S^*(X, \Omega)$. Since we are under the assumption  {\bf A3},  it follows  from Proposition~\ref{fejer} that  $\{p^k\}$ is bounded. Hence,  $\{p^k\}$  has a  cluster point $\bar p$.  Using Lemma~\ref{convergence}  we obtain that  ${\bar p} \in S^*(X, \Omega)$. Therefore, using again  Proposition~\ref{fejer} we conclude that  $\{p^k\}$  converges to ${\bar p} \in S^*(X, \Omega)$ and the theorem is proved.
\end{proof}
\section{Conclusions} \label{sec5}
Theorem~\eqref{eq:Conv} state that, if  the sequence $\{p^k\}$ generated by the  extragradient-type algorithm  is   finite and ends at iteration k, then $p^k$ is $\epsilon_k$-solution of $\mbox{VIP}(X,\Omega)$.  In fact, the concept of approximate solutions of $\mbox{VIP}(X,\Omega)$  can be  related with an important function, namely, the  {\it gap function}   $h:\Omega\rightarrow\mathbb{R}\cup\{+\infty\}$  defined by
\begin{equation}\label{gap}
h(p):=\sup_{q\in\Omega, v\in X(q)}\left\langle v,\exp^{-1}_qp\right\rangle.
\end{equation}
The relation between the function $h$ and the solutions of $\mbox{VIP}(X,\Omega)$ is given in the following lemma, which is a Riemannian version of the \cite[Lemma 4]{BurachikIusem1998}.
\begin{proposition} \label{pr:gfcss}
Let $h$ be the function defined in \eqref{gap}. Then, there holds $h^{-1}(0)=S^*(X, \Omega)$.
\end{proposition}
\begin{proof}
We will see first that a zero of $h$ is a solution of $\mbox{VIP}(X,\Omega)$, i.e,  $h^{-1}(0)\subset S^*(X, \Omega)$. Let  $p\in h^{-1}(0)$. Thus,  $h(p)=0$ and the definition of $h$  in \eqref{gap}  implies that
$$
\left\langle v,\exp^{-1}_qp \right\rangle\leq 0, \qquad q\in\Omega, \quad v\in X(q).
$$
On the other hand,  from the definition of the normal cone  $N_{\Omega}$ in \eqref{eq:nc},   we have
$$
\left\langle w,\exp^{-1}_qp \right\rangle\leq 0, \qquad q\in\Omega, \quad  w\in N_{\Omega}(q),
$$
Combining two last inequalities it is  easy to conclude that
$$
\left\langle 0-(v+w),\exp^{-1}_qp \right\rangle\geq 0, \qquad q\in\Omega, \quad v\in X(q), \quad w\in N_{\Omega}(q).
$$
Due to  Lemma~\ref{mon.cone}, the vector field  $X+N_{\Omega}$ is maximal monotone. Then, the maximality property together with latter inequality yields $0\in X(p)+N_{\Omega}(p)$, i.e.,  $p\in S^*(X,\Omega)$. Now, we are going to show that  the   solutions of $\mbox{VIP}(X,\Omega)$ are  zeros of $h$, i.e,  $S^*(X, \Omega)\subset h^{-1}(0)$. Suppose that $p\in S^*(X,\Omega)$. Then,  there exists $u\in X(p)$  such that
$$
\langle u,\exp^{-1}_pq\rangle\geq 0, \qquad q\in\Omega.
$$
Using the last inequality and the monotonicity of the vector field $X$ we obtain
$$
\langle v,\exp^{-1}_qp\rangle\leq 0, \qquad q\in\Omega, \quad v\in X(q).
$$
Therefore, definition of $h$ in \eqref{gap} implies $h(p)\leq 0$ and,   considering that  $h(p)\geq0$, we conclude that  $h(p)=0$, which ends  the proof.
\end{proof}
Let $\epsilon>0$. In view of Proposition~\ref{pr:gfcss} it make sense to define $\epsilon$-solution of $\mbox{VIP}(X,\Omega)$ as all point $\bar{p}\in \Omega$ such that
$$
h(\bar{p})=\sup_{q\in\Omega, v\in X(q)}\left\langle v,\exp^{-1}_q{\bar p}\right\rangle\leq \epsilon.
$$
We remark that if $M$ is a linear space, then the function gap  is convex. However, we do not know if this property is maintained in  Hadamard manifolds.


\section*{Acknowledgements}
The work  was supported by CAPES, FAPEG,  CNPq Grants   458479/2014-4,  471815/2012-8,   312077/2014-9,  305158/2014-7.
\bibliographystyle{habbrv}
\bibliography{Edvaldo}
\end{document}